\documentclass{scrartcl}
\usepackage[
color, 
noindent, 
english, 
enumstyle, 
theorem, 
operator, 
alphabet, 
cmd, 
arraytweak,
numberbysection,
arraytweak
]
{jgrygierek}
\usepackage[draft=false]{scrlayer-scrpage} 
\usepackage[
backend=biber, 
maxnames=42,
style=alphabetic,
citestyle=alphabetic,
doi=false,
url=false,
isbn=false,
giveninits=true
]{biblatex} 
\patchcmd{\thebibliography}{\chapter*}{\section*}{}{}
\allowdisplaybreaks
\date{\today}
\author{\texorpdfstring{Jens Grygierek\footnotemark[1]}{Jens Grygierek}}
\title{Multivariate Normal Approximation for functionals of random polytopes} 
\makeatletter
\hypersetup{
  pdftitle    = {\@title},
  pdfsubject  = {Mathematics/Probability, 52A22, 60D05, 60F05},
  pdfauthor   = {\@author},
  pdfkeywords = {\@title, \@author, central limit theorem, multivariate limit theorem, intrinsic volumes, f-vector, random polytope, random convex hull, stochastic geometry, Poisson point process, oracle estimator, volume estimation}
}
\makeatother

\begin{document}
\renewcommand{\thefootnote}{\fnsymbol{footnote}}

\maketitle	

\footnotetext[1]{Institute of Mathematics, Osnabr\"uck University, Germany. Email: jens.grygierek@uni-osnabrueck.de}
	
\begin{abstract}
	\noindent Consider the random polytope, that is given by the convex hull of a Poisson point process on a smooth convex body in $\RR^d$.
	We prove central limit theorems for continuous motion invariant valuations including the Will's functional and the intrinsic volumes of this random polytope. 
	Additionally we derive a central limit theorem for the oracle estimator, that is an unbiased an minimal variance estimator for the volume of a convex set.
	Finally we obtain a multivariate limit theorem for the intrinsic volumes and the components of the $\bff$-vector of the random polytope.
\end{abstract}

\section{Introduction}

We denote by $\cK^d$, $d \geq 2$, the collection of all compact convex sets in the Euclidean space $\RR^d$ and by $\cK^d_{sm}$ the set of all smooth convex bodies in $\cK^d$, which are all $K \in \cK^d$ that have nonempty interior, boundary of differentiability class $\cC^2$, and positive Gaussian curvature everywhere.
Let $\eta_t$ be a Poisson point process on $\RR^d$ with intensity measure $\mu = t \Lambda_d$, where $t > 0$ and $\Lambda_d$ denotes the $d$-dimensional Lebesgue measure on $\RR^d$, see e.g. \cite{LastPenrose2018} for more details on Poisson point processes.
We fix $K \in \cK^d_{sm}$ and investigate the random polytope $K_t \subset K$ defined as the convex hull of all points in $\eta_t \cap K$, 
\begin{align*}
	K_t := \conv\cbras*{x : x \in \eta_t \cap K}.
\end{align*}

The investigation of random convex hulls is one of the classical problems in stochastic geometry, see for instance the survey article \cite{Hug2013} and the introduction to stochastic geometry \cite{HugReitzner2016}.
Functionals like the intrinsic volumes $V_j(K_t)$ and the components $f_j(K_t)$ of the $\bff$-vector, see Section \ref{sec:geometric-prelim}, of the random polytope $K_t$ have been studied prominently, see \cites{CalkaYukich2014,Reitzner2010}[Section 1]{CalkaSchreiberYukich2013} and the references therein as well as the remarks and references on \cite[Theorem 5.5]{LachiezeReySchulteYukich2017} for more details.

Central limit theorems for $V_j(K_t)$ were proven in the special case that $K$ is the $d$-dimensional Euclidean unit ball, see \cite{CalkaSchreiberYukich2013} and \cite{Schreiber2002}. Short proofs for the binomial case $K_n$, where $n$ i.i.d. uniformly distributed points in $K$ are considered instead of a Poisson point process, were derived in \cite{ThaeleTurchiWespi2018}. 
Recently \cite{LachiezeReySchulteYukich2017} embedded the problem for both cases, the binomial and the Poisson case, in the theory of stabilizing functionals deriving central limit theorems for smooth convex bodies removing the logarithmic factors in the error of approximation improving the rate of convergence.

We extend the results of \cite{ThaeleTurchiWespi2018} on the intrinsic volumes to the more general case of continuous and motion invariant valuations. This includes the total intrinsic volume functional (Wills functional) $W(K_t)$ and a central limit theorem for the intrinsic volumes in the Poisson case similar to \cite[Theorem 1.1]{ThaeleTurchiWespi2018}.

We also obtain a univariate central limit theorem for the oracle estimator $\hat{\vartheta}_{oracle}$, that was derived in the remarkable work of \cite{BaldinReiss2016} on the estimation of the volume of a convex body.

Finally we investigate the components of the $\bff$-vector $f_j(K_t)$, $j \in \cbras*{0,\ldots,d-1}$ defined as the number of $j$-dimensional faces of $K_t$ and derive a multivariate limit theorem on the random vector composed of the intrinsic volumes and the $\bff$-vector.

\bigskip

For a continuous and motion invariant valuation $\varphi: \cK^d \rightarrow \RR$ we define the random variable $\varphi(K_t)$ and the corresponding standardization 
\begin{align}
	\tilde{\varphi}(K_t) := \frac{\varphi(K_t)-\E*{\varphi(K_t)}}{\sqrt{\V*{\varphi(K_t)}}}.
\end{align}

We prove a central limit theorem for $t \rightarrow \infty$ under some additional assumptions on the coefficients $c_i$ in the linear decomposition into the intrinsic volumes given by the remarkable theorem of Hadwiger stated here as Theorem \ref{thm:hadwiger}, see \cite{Chen2004, Klain1995, Hadwiger1957} for more details and different proofs.
Given this restriction of our model to valuation functionals $\varphi(K_t)$, that can be safely assumed to not loose variance compared to the intrinsic volume functionals $V_j(K_t)$, we show the first main result of this paper:

We denote by $\dist_W$ the Wasserstein distance, see \eqref{eq:metric-wasserstein} in Section \ref{sec:malliavin-stein:uni} for the precise Definition.

\begin{theorem}[Univariate Limit Theorem]\label{thm:univariate}
	Assume that $\varphi$ is a continuous and motion invariant valuation, with linear decomposition given by Hadwiger \eqref{eq:hadwiger}, 
	such that $c_i c_j \geq 0$ for all $i,j \in \cbras*{0, \ldots, d}$ and suppose that at least one index $k \in \cbras*{1, \ldots, d}$ exists, such that $c_k \neq 0$.
	Then
	\begin{align}
		\dist_W(\tilde{\varphi}(K_t), Z) = \cO\rbras*{t^{-\frac{1}{2} + \frac{1}{d+1}} \log(t)^{3 + \frac{2}{d+1}}},
	\end{align}
	where $Z \distribute \cN(0,1)$.
\end{theorem}

For $j \in \cbras*{1, \ldots, d}$ we can directly obtain the central limit theorem for the standardized $j$-th intrinsic volume in the Poisson model by setting $\varphi(K_t) = V_j(K_t)$.

Further we can directly obtain a central limit theorem with rate of convergence for the total intrinsic volume, also known as the Wills functional, see \cite{Wills1973,Hadwiger1975,McMullen1975}, setting  the coefficients $c_j = 1$, $j \in \cbras*{0,\ldots,d}$ in Theorem \ref{thm:univariate}:

\begin{corollary}[Wills functional]
	Let $W(K_t)$ denote the Wills functional, defined by $W(K_t) = \sum_{j=0}^d V_j(K_t)$ and denote by $\tilde{W}(K_t)$ the corresponding standardization.
	Then 
	\begin{align}
		\dist_W(\tilde{W}(K_t), Z) = \cO\rbras*{t^{-\frac{1}{2} + \frac{1}{d+1}} \log(t)^{3 + \frac{2}{d+1}}},
	\end{align}
	where $Z \distribute \cN(0,1)$.
\end{corollary}

In the remarkable work \cite{BaldinReiss2016} on the estimation of the volume of a convex body $V_d(K)$ given the intensity $t > 0$ is known, the oracle estimator 
\begin{align*}
	\hat{\vartheta}_{oracle}(K_t) = V_d(K_t) + \frac{f_0(K_t)}{t},
\end{align*}
is derived.
This estimator is unbiased,
\begin{align*}
	\E*{\hat{\vartheta}_{oracle}(K_t)} = V_d(K),
\end{align*}
and of minimal variance among all unbiased estimators (UMVU), see \cite[eq. (3.2), Theorem 3.2]{BaldinReiss2016}. Additionally the variance can be obtained by combining \cite[Theorem 3.2]{BaldinReiss2016} with \cite[Lemma 1]{Reitzner2005a} yielding

\begin{align}\label{eq:oracle-variance}
	\V*{\hat{\vartheta}_{oracle}(K_t)} = \frac{1}{t} \E*{V_d(K \setminus K_t)} = \gamma_d \Omega(K)(1+o(1))t^{-1-\frac{2}{d+1}},
\end{align}
for $t \rightarrow \infty$, where the constant $\gamma_d$ only depends on the dimension and is known explicitly and $\Omega(K)$ denotes the affine surface area of $K$.
This enables us to prove the following univariate limit theorem for the oracle estimator:

\begin{theorem}[Oracle estimator]\label{thm:oracle}
	Let $\hat{\vartheta}_{oracle}$ be the oracle estimator for a smooth convex body $K \in \cK^d_{sm}$ and denote by $\tilde{\vartheta}_{oracle}$ the corresponding standardization. 
	Then
	\begin{align}
		\dist_W(\tilde{\vartheta}_{oracle}(K_t),Z) = \cO\rbras*{t^{-\frac{1}{2}+\frac{1}{d+1}} \log(t)^{3d+\frac{2}{d+1}}},
	\end{align}
	where $Z \distribute \cN(0,1)$.
\end{theorem}

Finally we provide a multivariate limit theorem for the intrinsic volumes and the components of the $\bff$-vector of the random polytope $K_t$, that were considered in \cite[Theorem 5.5]{LachiezeReySchulteYukich2017} in the univariate case.

We denote by $\dist_3$ the distance given by \eqref{eq:metric-3} in Section \ref{sec:malliavin-stein:multi}.

\begin{theorem}[Multivariate Limit Theorem]\label{thm:multivariate}
	Let
	\begin{align}\label{eq:multivariate-functional}
		F(K_t) : = \rbras*{\tilde{V}_1(K_t), \ldots, \tilde{V}_d(K_t), \tilde{f}_0(K_t), \ldots, \tilde{f}_{d-1}(K_t)} \in \RR^{2d}
	\end{align}		
	be the vector of the standardized intrinsic volumes $\tilde{V}_j$ and standardized number of $k$-dimensional faces $\tilde{f}_k(K_t)$, i.e.
	\begin{align}
		\tilde{V}_j(K_t) := \frac{V_j(K_t)-\E*{V_j(K_t)}}{\sqrt{\V*{V_j(K_t)}}} 
		\quad 
		\text{and}
		\quad
		\tilde{f}_k(K_t) := \frac{f_k(K_t)-\E*{f_k(K_t)}}{\sqrt{\V*{f_k(K_t)}}}.
	\end{align}	 
	We denote by $F_i := F_i(K_t)$, $i = 1, \ldots, 2d$ the $i$-th component of the multivariate functional $F(K_t)$.
	Define $\Sigma(t) = (\sigma_{ij}(t))_{i,j \in \cbras*{1, \ldots, 2d}}$ as the covariance matrix of $F(K_t)$, i.e. $\sigma_{ij}(t) := \Cov*{F_i}{F_j}$ and $\sigma_{ii}(t) = 1$ for all $i \in \cbras*{1, \ldots, 2d}$ and all $t > 0$.
	Then
	\begin{align}
		\dist_3(F(K_t), N_{\Sigma(t)}) = \cO\rbras*{t^{-\frac{1}{2}+\frac{1}{d+1}}\log(t)^{3d+\frac{2}{d+1}}}
	\end{align}
	where $N_{\Sigma(t)} \distribute \cN(0, \Sigma(t))$.
\end{theorem}

Note that $N_{\Sigma(t)}$ still depends on the intensity $t$. 
This gives rise to the following questions:

\bigskip

\textbf{Open Problems:}

The limit of the variances and co-variances are still unknown, thus we can not set $\sigma_{ij}$ to be the limit of the correlations (rescaled co-variances) $\sigma_{ij}(t)$.
These limits would give rise to a limit theorem providing a fixed multivariate Gaussian distribution $\cN(0, \Sigma)$ with fixed co-variance matrix $\Sigma$. In this case, the rate of the limit theorem will also contain the rate of the correlations on the right hand side of the bound, thus it would be beneficial to obtain the limit $\sigma_{ij}$ of $\sigma_{ij}(t)$ including an upper bound for $\abs{\sigma_{ij}(t) - \sigma_{ij}}$, since the convergence of the correlations could be slower than the rate given in Theorem \ref{thm:multivariate}, slowing down the overall convergence. 
We should mention, that Calka, Schreiber and Yukich, were able to derive limits for the variance in the case that $K$ is the euclidean unit ball using the paraboloid growth process, see \cite{CalkaSchreiberYukich2013}, but up to our knowledge there are no results on the limit of the variance in a general (smooth) convex body and also no results on the co-variances of $F(K_t)$.

Note that the Euler-Poincar\'{e} equations and more general the Dehn-Sommerville equations, see \cite[Chapter 8]{Ziegler1995},  imply linear dependencies on the components of the $\bff$-vector. Thus the covariance matrix $\Sigma(t)$ is singular and therefore $\rank*{\Sigma(t)} < 2d$, which gives rise to the question what $\rank*{\Sigma(t)}$ actually is and if this also applies to the limiting co-variance matrix $\Sigma$?

\begin{remark}
	Note that the univariate results can be derived for the $d_3$-distance using the multivariate result. Since the additional work that is needed to prove the univariate results alongside the multivariate limit theorem is small, we decided to directly proof the univariate results in the Wasserstein distance.
\end{remark}

\bigskip

The paper is organized as follows. For the convenience of the reader, we repeat the relevant material on the Malliavin-Stein-Method for normal approximation of Poisson functionals in Section 2. In Section 3 we introduce some background material on convex geometry and corresponding estimations using floating bodies without proofs, to keep our presentation reasonably self-contained.
The proofs of our main results are presented in Section 4, starting with the central limit theorem for valuations, Theorem \ref{thm:univariate}, in Subsection 4.1. handling the intrinsic volumes. In Subsection 4.2 we prove the multivariate limit theorem, Theorem \ref{thm:multivariate} by extending our proof to the components of the $\bff$-vector. Finally we can combine the results derived in the proofs before to obtain the central limit theorem for the oracle estimator, Theorem \ref{thm:oracle}.

\section{Stochastic preliminaries}

Let $\eta$ be a Poisson point process over the Euclidean space $(\RR^d, \sB^d)$ with intensity measure $\mu$. 
One can think of $\eta$ as a random element in the space $\rN_\sigma$ of all $\sigma$-finite counting measures $\chi$ on $\RR^d$, i.e. $\chi(B) \in \NNN \cup \{ \infty \}$ for all $B \in \sX$, where the space $\rN_\sigma$ is equipped with the $\sigma$-field $\sN_\sigma$ generated by the mappings $\chi \rightarrow \chi(B)$, $B \in \sB^d$.
To simplify our notation we will often handle $\eta$ as a random set of points given by
\begin{align*}
	x \in \eta \Leftrightarrow x \in \cbras*{y \in \RR^d : \eta(\cbras{y}) > 0}.
\end{align*}
We call a random variable $F$ a Poisson functional if there exists a measurable map $f:\rN_\sigma \rightarrow \RR$ such that $F = f(\eta)$ almost surely.
The map $f$ is called the representative of $F$.
We define the difference operator or so called ``add-one-cost operator'':

\begin{definition}
	Let $F$ be a Poisson functional and $f$ its corresponding representative, then the first order difference operator is given by
	\begin{align*}
		D_xF := f(\eta + \delta_x) - f(\eta), \quad x \in \RR^d,
	\end{align*}
	where $\delta_x$ denotes the Dirac measure with mass concentrated in $x$.
	We say that $F$ belongs to the domain of the difference operator, i.e. $F \in \dom{D}$, if $\E*{F^2} < \infty$ and
	\begin{align*}
		\int\limits_{\RR^d} \! \E*{(D_xF)^2} \, \mu(\id x) < \infty.
	\end{align*}
	The second order difference operator is obtained by iterating the definition:
	\begin{align*}
		D_{x_1,x_2}^2F & := D_{x_1}(D_{x_2}F)\\
			& \phantom{:}= f(\eta + \delta_{x_1} + \delta_{x_2}) - f(\eta + \delta_{x_1}) - f(\eta + \delta_{x_2}) + f(\eta), \quad x_1, x_2 \in \RR^d.
	\end{align*}
\end{definition}

For a depper discussion of the underlying theory of Poisson point processes, Malliavin-Calculus, the Wiener-It\^{o}-Chaos Expansion and the Malliavin-Stein Method see \cite{PeccatiReitzner2016} and \cite{LastPenrose2018}.

\subsection{Malliavin-Stein-Method for the univariate case}\label{sec:malliavin-stein:uni}

We denote by $\Lip{1}$ the class of Lipschitz functions $h:\RR \rightarrow \RR$ with Lipschitz constant less or equal to one.
Given two $\RR$-valued random variables $X,Y$, with $\Eabs*{X} < \infty$ and $\Eabs*{Y} < \infty$ the Wasserstein distance between the laws of $X$ and $Y$, written $\dist_W(X,Y)$ is defined as
\begin{align}\label{eq:metric-wasserstein}
	\dist_W(X,Y) := \sup\limits_{h \in \Lip{1}} \abs*{\E*{h(X)} - \E*{h(Y)}}.
\end{align}

Note that if a sequence $(X_n)$ of random variables satisfies $\lim_{n \rightarrow \infty} \dist_W(X_n, Y) = 0$, then it holds that $X_n$ converges to $Y$ in distribution, see \cite[p. 219 and Proposition B.9]{LastPenrose2018}. Especially if $Y$ has standard Gaussian distribution, we obtain a central limit theorem by showing $\dist_W(X_n, Y) \rightarrow 0$, which we will achieve by rephrasing the bound given by \cite[Theorem 1.1]{LastPeccatiSchulte2016} which is an extension of \cite[Theorem 3.1]{PeccatiSoleTaqquUtzet2010}, see also \cite[Chapter 21.1, 21.2]{LastPenrose2018} for a slightly different form and proofs.

\begin{theorem}\label{thm:bound-uni}
	Let $F \in \dom{D}$ be a Poisson functional such that $\E*{F} = 0$ and $\V*{F} = 1$.
	Define
	\begin{align*}
		\tau_1 & := \int\limits_{K^3} \! 
				\rbras*{\E*{(D_{x_1,x_3}^2F)^4}\E*{(D_{x_2,x_3}^2 F)^4} \E*{(D_{x_1}F)^4}\E*{(D_{x_2}F)^4}}^{\frac{1}{4}} 
				\mu^3(\id (x_1, x_2, x_3))\\
		\tau_2 & := \int\limits_{K^3} \! 
				\rbras*{\E*{(D_{x_1,x_3}^2 F)^4}\E*{(D_{x_2,x_3}^2 F)^4}}^{\frac{1}{2}}
				\mu^3(\id (x_1, x_2, x_3))\\
		\tau_3 & := \int\limits_{K}
				\Eabs*{D_xF}^3
				\mu(\id x)
	\end{align*}
	and let $Z$ be a standard Gaussian random variable,
	then
	\begin{align*}
		\dist_W(F,Z) \leq 2 \sqrt{\tau_1} + \sqrt{\tau_2} + \tau_3.
	\end{align*}
\end{theorem}

\subsection{Malliavin-Stein-Method for the multivariate case}\label{sec:malliavin-stein:multi}

We denote by $\cH_m$ the class of all $\cC^3$-functions $h:\RR^m \rightarrow \RR$ such that all absolute values of the second and third partial derivatives are bounded by one, i.e.
\begin{align*}
	\max\limits_{1 \leq i_1 \leq i_2 \leq m} \sup\limits_{x \in \RR^d} \abs*{\frac{\partial^2}{\partial x_{i_1} \partial x_{i_2}} h(x)} \leq 1 \quad \text{and} \quad
	\max\limits_{1 \leq i_1 \leq i_2 \leq i_3 \leq m} \sup\limits_{x \in \RR^d} \abs*{\frac{\partial^3}{\partial x_{i_1} \partial x_{i_2} \partial x_{i_3}} h(x)} \leq 1.
\end{align*}
Given two $\RR^m$-valued random vectors $X,Y$ with $\dsE\norm*{X}^2 < \infty$ and $\dsE\norm{Y}^2 < \infty$ the distance $\dist_3$ between the laws of $X$ and $Y$, written $\dist_3(X,Y)$ is defined as
\begin{align}\label{eq:metric-3}
	\dist_3(X,Y) := \sup\limits_{h \in \cH_m}\abs*{\E*{h(X)} - \E*{h(Y)}}.
\end{align} 

Note that if a sequence $(X_n)$ of random vectors satisfies $\lim_{n \rightarrow \infty} \dist_3(X_n, Y) = 0$, then it holds that $X_n$ converges to $Y$ in distribution, see \cite[Remark 2.16]{PeccatiZheng2010}. Especially if $Y$ is a $m$-dimensional centered Gaussian vector with covariance matrix $\Sigma \in \RR^{m \times m}$, we obtain a multivariate limit theorem by showing $\dist_3(X_n, Y) \rightarrow 0$.
We will achieve this, similar to the univariate central limit theorem, by rephrasing the bound given by \cite[Theorem 1.1]{SchulteYukich2018} which extends \cite{PeccatiZheng2010}, to provide the multivariate analogon to the univariate result derived in \cite{LastPeccatiSchulte2016}, that was stated here as Theorem \ref{thm:bound-uni} above.

\begin{theorem}\label{thm:bound-multi}
	Let $F = (F_1, \ldots, F_m)$, $m \geq 2$, be a vector of Poisson functionals $F_1, \ldots, F_m \in \dom{D}$ with $\E*{F_i} = 0$, $i \in \cbras*{1, \ldots, m}$.
	Define
	\begin{align*}
		\gamma_1 & := \sum\limits_{i,j=1}^m \int\limits_{K^3} \! 
			\rbras*{\E*{(D_{x_1,x_3}^2F_i)^4}\E*{(D_{x_2,x_3}^2 F_i)^4}\E*{(D_{x_1}F_j)^4}\E*{(D_{x_2}F_j)^4}}^{\frac{1}{4}} 
			\mu^3(\id (x_1, x_2, x_3))\\
		\gamma_2 & := \sum\limits_{i,j=1}^m \int\limits_{K^3} \! 
			\rbras*{\E*{(D_{x_1,x_3}^2F_i)^4}\E*{(D_{x_2,x_3}^2 F_i)^4}\E*{(D_{x_1,x_3}^2F_j)^4}\E*{(D_{x_2,x_3}^2 F_j)^4}}^{\frac{1}{4}} 
			\mu^3(\id (x_1, x_2, x_3))\\
		\gamma_3 & := \sum\limits_{i=1}^m \int\limits_{K}
			\Eabs*{D_xF_i}^3
			\mu(\id x)
	\end{align*}
	and let $\Sigma = (\sigma_{ij})_{i,j \in \cbras*{1, \ldots, m}} \in \RR^{m \times m}$ be positive semi-definite,
	then
	\begin{align*}
		\dist_3(F,N_\Sigma) \leq \frac{m}{2} \sum\limits_{i,j=1}^m \abs*{\sigma_{ij} - \Cov*{F_i}{F_j}} + m \sqrt{\gamma_1} + \frac{m}{2} \sqrt{\gamma_2} + \frac{m^2}{4} \gamma_3.
	\end{align*}
\end{theorem}

\section{Geometric preliminaries}\label{sec:geometric-prelim}

Fix $j \in \cbras*{1, \ldots, d}$ and let $G(d,j)$ denote the Grassmannian of all $j$-dimensional linear subspaces of $\RR^d$ equipped with the uniquely determined Haar probability measure $\nu_j$, see \cite[Section 4.4]{Schneider2014}. For $k \in \NNN$ the $k$-dimensional volume of the $k$-dimensional unit ball $\BB^k$ is denoted by $\kappa_k := \pi^{\frac{k}{2}}\Gamma\rbras*{1+\frac{k}{2}}^{-1}$.

For a convex body $K \in \cK^d$ the $j$-dimensional Lebesgue measure of the orthogonal projection of $K$ onto the linear subspace $\bbL \in G(d,j)$ is denoted by $\Lambda_j(K \vert \bbL)$.

For $j \in \cbras*{1,\ldots,d}$, the $j$-th intrinsic volume of $K$ is given by Kubota's formula, see \cite[p. 222]{SchneiderWeil2008}:
\begin{align}\label{eq:kubota}
	V_j(K) = \binom{d}{j} \frac{\kappa_d}{\kappa_j \kappa_{d-j}} \int\limits_{G(d,j)} \Lambda_j(K \vert \bbL) \nu_j(\id \bbL)
\end{align}
and for $j = 0$ the $0$-th intrinsic volume of $K$, $V_0(K)$ is the Euler characteristics of $K$, therefore we have $V_0(K) = \1\cbras*{K \neq \emptyset}$.
It is worth mentioning, that $V_1(K)$ is the mean with, $V_{d-1}(K)$ is the surface area up to multiplicative constants not depending on $K$ and $V_d(K)$ equals the $d$-dimensional Lebesgue-volume of $K$. The intrinsic volumes are crucial examples of continuous, motion invariant valuations:

\begin{definition}
	A real function on the space of convex bodies, $\varphi: \cK^d \rightarrow \RR$ , is called a valuation, if and only if
	\begin{align}
		\label{eq:valuation-property}
		\varphi(K) + \varphi(L) = \varphi(K \cup L) + \varphi(K \cap L)
	\end{align}
	holds, whenever $K, L, K \cup L \in \cK^d$.
	It is called continuous, if it is continuous according to the Hausdorff metric on $\cK^d$, and it is called invariant under rigid motions if it is invariant under translations and rotations on $\RR^d$.
\end{definition}

The theorem of Hadwiger \cite{Hadwiger1957, Klain1995, Chen2004} states, that every continuous and motion invariant valuation $\varphi: \cK^d \rightarrow \RR$ can be decomposed into a linear combination of intrinsic volumes:

\begin{theorem}[Hadwiger]\label{thm:hadwiger}
	Let $\varphi$ be a continuous and motion invariant valuation. 
	Then there existing coefficients $c_i \in \RR$, $i \in \cbras*{0, \ldots, d}$, such that for all convex sets $L \in \cK^d$ it holds, that 
	\begin{align}\label{eq:hadwiger}
		\varphi(L) = \sum\limits_{i = 0}^d c_i V_i(L),
	\end{align}
	where $V_i$ denotes the $i$-th dimensional intrinsic volume.
\end{theorem}

For further information on Hadwiger's theorem, convex geometry and integral geometry we refer the reader to \cite{Gruber2007, SchneiderWeil2008, Schneider2014}.


Let $P \in \cK^d$ be a polytope and $i \in \cbras*{0, \ldots, d}$. 
We denote by $\cF_i(P)$ the set of all $i$-dimensional faces, $i$-faces for short, of $P$ and by $\cF_i^{\Vis}(P,x)$ the restriction to those $i$-faces that can be seen from $x$, where we consider a face $\frF$ of $P$ to be seen by $x$ if all points $z \in \frF$ can be connected by a straight line $[z,x]$ to $x$ such that the intersection of this line with $P$ only contains the starting point $z$, i.e.
\begin{align*}
	\cF_i^{\Vis}(P,x) := \cbras*{\frF \in \cF_i(P) : \forall z \in \frF : [x,z] \cap P = \cbras*{z}}.
\end{align*}
The sets of all faces resp. all visible faces are denoted by $\cF(P) := \bigcup_{i=0}^d \cF_i(P)$ resp. $\cF^{\Vis}(P,x) := \bigcup_{i=0}^d \cF_i^{\Vis}(P,x)$.

For a vertex $v \in \cF_0(P)$ the link of $v$ in $P$ is the set of all faces of $P$ that do not contain $v$ but are contained in a (higher dimensional) face that contains $v$, i.e.
\begin{align*}
	\link(P, v) := \cbras*{\frF \in \cF(P) : v \not\in \frF \text{ and } \exists \frG \in \cF(P) : \frF \subset \frG \ni v},
\end{align*}
see \cite[Chapter 8.1]{Ziegler1995} for a recent account of the theory.

The number of $i$-dimensional faces of $P$ will be denoted by $f_i(P)$, i.e.
\begin{align*}
	f_i(P) = \abs{\cF_i(P)}.
\end{align*}
Note that the vector $(f_{-1}(P), f_0(P), \ldots, f_d(P))$ with $f_{-1}(P) := V_0(P)$ is the $\bff$-vector of $P$, see \cite[Definition 8.16, p. 245]{Ziegler1995} for more details.

\subsection{Geometric estimations}

We introduce the notion of the $\varepsilon$-floating body, following \cite[Section 2.2.3]{Reitzner2010}.
For a fixed $K \in \cK^d$ and a closed halfspace $H$ we call the intersection $C = H \cap K$ a cap of $K$. If $C$ has volume $V_d(C) = \varepsilon$, we call $C$ an $\varepsilon$-cap of $K$.
We define the function $v: K \rightarrow \RR$ by
\begin{align*}
	v(z) = \min\cbras*{V_d(K \cap H) : H \text{ is a halfspace with } z \in H},
\end{align*}
and the floating body with parameter $\varepsilon$, $\varepsilon$-floating body for short, as the level set
\begin{align*}
	K(v \geq \varepsilon) = \cbras*{z \in K : v(z) \geq \varepsilon},
\end{align*}
which is convex, since it is the intersection of halfspaces.
The wet part of $K$ is defined as $K(v \leq \varepsilon)$, where the name comes from the $3$-dimensional picture when $K$ is a box containing $\varepsilon$ units of water.
Note that the $\varepsilon$-floating body is (up to its boundary) the remaining set of $K$, if all $\varepsilon$-caps are removed from $K$ and the wet part is the union of these caps.
For the convenience of the reader, we will only use the notation for the floating body $K(v \geq \varepsilon)$ to prevent confusion with the wet part denoted by $K(v \leq \varepsilon)$.
From now on, we will assume that the parameter $\varepsilon > 0$ is sufficiently small. Thus we can use the following lemmas from \cite[Lemma 6.1-6.3]{Vu2005}:

\begin{lemma}\label{lem:geo:diam}
	Let $C$ be an $\varepsilon$-cap of $K$, then there are two constants $c_1, c_2 \in (0,\infty)$ such that the diameter of $C$, $\diameter(C) = \sup_{x,y \in C}\norm{x-y}$, is bounded by
	\begin{align*}
		c_1 \varepsilon^{\frac{1}{d+1}} \leq \diameter(C) \leq c_2 \varepsilon^{\frac{1}{d+1}}.
	\end{align*}
\end{lemma}

\begin{lemma}\label{lem:geo:vol}
	Let $x$ be a point on the boundary $\partial K$ of $K$ and $D(x,\varepsilon)$ the set of all points on the boundary which are of distance at most $\varepsilon$ to $x$. Then the convex hull of $D(x,\varepsilon)$ has volume at most $c_3 \varepsilon^{d+1}$, where $c_3 \in (0,\infty)$ is some constant not depending on $\varepsilon$.
\end{lemma}

\begin{lemma}\label{lem:geo:union}
	Let $C$ be an $\varepsilon$-cap of $K$. The union of all $\varepsilon$-caps intersecting $C$ has volume at most $c_4 \varepsilon$, where $c_4 \in (0,\infty)$ is some constant not depending on $\varepsilon$.
\end{lemma}

\section{Proofs of the main results}

To shorten our notation we write $K^x_t$ resp. $K^y_t$ for the convex hull of $(\eta_t + \delta_x) \cap K$ resp. $(\eta_t + \delta_y) \cap K$ and $K^{xy}_t$ for the convex hull of $(\eta_t + \delta_x + \delta_y) \cap K$.
Further we will use $\bC \in (0,\infty)$ to denote a constant, that can depend on the dimension and the convex set $K$ but is independent of the intensity of our Poisson point process $t$. For sake of brevity we will not mention this properties of $\bC$ in the following, additionally the value of $\bC$ may also differ from line to line.
We will use $g(t) \ll f(t)$ to indicate that $g(t)$ is of order at most $f(t)$, i.e.
\begin{align*}
	g(t) \ll f(t) & :\Leftrightarrow g(t) = \cO\rbras*{f(t)}\\
		&  \phantom{:}\Leftrightarrow \exists c > 0, t_0 > 0 : \forall t > t_0 : g(t) \leq c f(t),
\end{align*}
where $c$ and $t_0$ are constants not depending on $t$. We will use $g(t) = \Theta(f(t))$ to indicate that $g(t)$ is of the same order of $f(t)$, i.e.
\begin{align*}
	g(t) = \Theta(f(t))  :\Leftrightarrow f(t) = \cO\rbras*{g(t)} \text{~and~} g(t) = \cO\rbras*{f(t)}.
\end{align*}

For sufficiently large $t > 0$, we define $\varepsilon_t := c \tfrac{\log(t)}{t}$ with $c > 0$ and denote by $K(v \geq \varepsilon_t)$ the $\varepsilon_t$-floating body of $K$.
Let $A(\varepsilon_t, t) := \cbras*{K(v \geq \varepsilon_t) \subseteq K_t}$ be the event, that the $\varepsilon_t$-floating body is contained in the random polytope $K_t$. Recall the well known Lemma from \cites{BaranyDalla1997}{Vu2005}[Lemma 2.2]{Reitzner2010} in a slightly modified version for the Poisson case:

\begin{lemma}\label{lem:floating-body-contained}
	For any $\beta \in (0,\infty)$ there exists a constant $c(\beta, d) \in (0,\infty)$ only depending on $\beta$ and the space dimension $d$, such that the event $A(\varepsilon_t,t)$, that the $\varepsilon_t$-floating body is contained in the random polytope $K_t$ occurs with high probability. More precisely, the probability of the complementary event $A^c(\varepsilon_t, t)$ has polynomial decay with exponent $-\beta$ for $t \rightarrow \infty$, i.e.
	\begin{align*}
		\Prob*{A^c(\varepsilon_t, t)} < \bC t^{-\beta}.
	\end{align*}
	whenever $t$ is sufficiently large.
\end{lemma}

Note that the parameter $\beta$ can be freely chosen in $(0, \infty)$, thus for $\beta = 16d+1$, which is sufficiently large for all our purposes, we get $c(\beta, d)$ and therefore we can define $\varepsilon_t$ such that $K(v \geq \varepsilon_t) \subseteq K_t$ with high probability according to Lemma \ref{lem:floating-body-contained}.

We will use the following estimation of subsets of $G(d,j)$ from \cite[Lemma 1]{BaranyFodorVigh2010} to handle the projections arising from  Kubota's formula in our proof of Theorem \ref{thm:univariate}:

\begin{lemma}\label{lem:measure-of-angle-set}
	For $z \in \bbS^{d-1}$ and $\bbL \in G(d,j)$ we define the angle $\sphericalangle(z,\bbL)$ as the minimum of all angles $\sphericalangle(z,x)$, $x \in \bbL$.
	For sufficiently small $\alpha > 0$ one has that
	\begin{align*}
		\nu_j\rbras*{\cbras*{\bbL \in G(d,j) : \sphericalangle(z,\bbL) \leq \alpha}} = \Theta\rbras*{a^{d-j}}.
	\end{align*}
\end{lemma}

\subsection{Proof of Theorem \ref{thm:univariate}: valuation functional}

We first recall that the valuation functional $\varphi(K_t)$ can be decomposed with Hadwiger \eqref{eq:hadwiger} into the linear combination of intrinsic volumes, thus the variance $\V*{\varphi(K_t)}$ can be rewritten into
\begin{align*}
	\V*{\varphi(K_t)} = \sum\limits_{i = 0}^d c_i^2 \V*{V_i(K_t)} + 2 \sum\limits_{i=0}^d \sum\limits_{j = i+1}^d c_i c_j \Cov*{V_i(K_t)}{V_j(K_t)}.
\end{align*}

For $V_i$, $i \in \cbras*{1,\ldots,d}$ we will use the variance bound from \cite[eq. 5.20, eq. 5.22, eq. 5.23]{LachiezeReySchulteYukich2017}, see also \cite[Corollary 7.1]{CalkaSchreiberYukich2013} and \cite{Reitzner2005a}:
\begin{align}\label{eq:intrinsic:variance}
	t^{-1-\frac{2}{d+1}} \ll \V*{V_i(K_t)} \ll t^{-1-\frac{2}{d+1}}.
\end{align}

Since $V_0(K_t)$ is the Euler characteristics of $K_t$ we have $V_0(K_t) = \1\cbras*{K_t \neq \emptyset}$ and therefore $V_0(K_t)$ is a Bernoulli distributed random variable with success probability $\Prob*{V_0(K_t) = 1} = 1 - e^{-t \Lambda_d(K)}$. The expectation is given by $\E*{V_0(K_T)} = 1 - e^{-t \Lambda_d(K)}$ and the variance by $\V*{V_0(K_t)} = (1 - e^{-t \Lambda_d(K)})e^{-t \Lambda_d(K)}$, which can be bounded by
\begin{align}\label{eq:0-th:variance}
	0 \leq \V*{V_0(K_t)} \ll e^{-t \Lambda_d(K)} \ll t^{-1-\frac{2}{d+1}}.
\end{align}

\begin{lemma}\label{lem:intrinsic-volumes-covariances}
	For all $i,j \in \cbras*{0, \ldots, d}$ the intrinsic volumes of $K_t$ are non negatively correlated and their covariances are bounded from above by the same order of magnitude as the variance, i.e. 
	\begin{align}\label{eq:intrinsic-volumes-covariances}
		0 \leq \Cov*{V_i(K_t)}{V_j(K_t)} \ll t^{-1-\frac{2}{d+1}}.
	\end{align}
\end{lemma}

\begin{proof}
	Since $D_xV_j(K_t) \geq 0$ for all $x \in \RR^d$ it follows from the Harris-FKG inequality for Poisson processes, see \cite[Theorem 11]{Last2016}, that $\E*{V_i(K_t) V_j(K_t)} \geq \E*{V_i(K_t)}\E*{V_j(K_t)}$, which directly implies the lower bound on the covariances.
	The Cauchy-Schwarz inequality implies 
	\begin{align*}
		\Cov*{V_i(K_t)}{V_j(K_t)} \leq \sqrt{\V*{V_i(K_t)} \V*{V_j(K_t)}},
	\end{align*}
	thus, using \eqref{eq:intrinsic:variance} and \eqref{eq:0-th:variance}, the upper bound on the covariances is obtained.
\end{proof}

We are now in a position to bound the variance of our valuation functional with the following Lemma:

\begin{lemma}
	Under the assumptions of Theorem \ref{thm:univariate} the variance of the valuation functional is bounded by
	\begin{align}
		t^{-1-\frac{2}{d+1}} \ll \V*{\varphi(K_t)} \ll t^{-1-\frac{2}{d+1}}.
	\end{align}
\end{lemma}

\begin{proof}
	We assumed $c_ic_j \geq 0$ for all $i,j$ and that there exists at least one index $k \in \cbras*{1,\ldots,d}$ such that $c_k \neq 0$. Thus Lemma \ref{lem:intrinsic-volumes-covariances} implies
	\begin{align*}
		\V*{\varphi(K_t)} \geq c_k^2 \V*{V_k(K_t)} \gg t^{-1-\frac{2}{d+1}},
	\end{align*}
	and 
	\begin{align*}
		\V*{\varphi(K_t)} \ll (d+1) t^{-1-\frac{2}{d+1}} + (d+1)d \sqrt{\rbras*{t^{-1-\frac{2}{d+1}}}\rbras*{t^{-1-\frac{2}{d+1}}}} \ll t^{-1-\frac{2}{d+1}},
	\end{align*}
	which completes the proof.
\end{proof}

The crucial part in the proof of Theorem \ref{thm:univariate} is the application of the general bound given by Theorem \ref{thm:bound-uni}, thus we need to investigate the moments occuring in $\tau_1$, $\tau_2$ and $\tau_3$. In the first step, we adapt and slightly extend the proof from \cite{ThaeleTurchiWespi2018} for the binomial case, to work in the Poisson case, yielding upper bounds on the moments of the first and second order difference operators applied to the intrinsic volumes $V_j(K_t)$ which will be used in the second step to derive the bounds for the valuation functional.

\textbf{First order difference operator:}

Fix $x \in K$ and $j \in \cbras{1, \ldots, d}$, then conditioned on the event $A(\varepsilon_t, t)$, it follows that
\begin{align}\label{eq:first-order-diff-conditioned}
	D_xV_j(K_t) = \1\cbras*{x \in K \setminus K_t} D_xV_j(K_t) = \1\cbras*{x \in K \setminus K(v \geq \varepsilon_t)} D_xV_j(K_t),
\end{align}
thus we can restrict the following to the case $x \in K \setminus K(v \geq \varepsilon_t)$.

For $x \in K \setminus K(v \geq \varepsilon_t)$ we define $z$ to be the closest point to $x$ on the boundary $\partial K$. Since $K$ is smooth $z$ is uniquely determined, if $\varepsilon_t$ is sufficiently small.

The visible region of $z$ is defined as the set of all points that can be connected to $z$ by a straight line in $K$ avoiding $K(v \geq \varepsilon_t)$, i.e.
\begin{align}
	\Vis_z(t) := \cbras*{y \in K \setminus K(v \geq \varepsilon_t) : [y,z] \cap K(v \geq \varepsilon_t) = \emptyset}.
\end{align}
Note that, given the sandwiching $K(v \geq \varepsilon_t) \subset K_t \subset K$, a random point $x$ can influence the random polytope only in the visibility region.

We construct a full-dimensional spherical cap $C$ such that $K^x_t \setminus K_t \subseteq C$.
The definition of the visible region, that was first used in \cite{Vu2005} and \cite{BaranyReitzner2010}, is crucial in the following steps:

Let $y_1,y_2 \in \Vis_z(t)$, then there existing two $\varepsilon_t$-caps $C_1$ and $C_2$ such that the straight line $[y_1, z]$ resp. $[y_2, z]$ is contained in $C_1$ resp. $C_2$, thus
\begin{align*}
	\norm{y_1-y_2} \leq \norm{y_1 - z} + \norm{y_2 - z} \leq \diameter(C_1) + \diameter(C_2).
\end{align*}

Since the diameter of any $\varepsilon_t$-cap $C$ of $K$ can be bounded by $\bC \varepsilon_t^\frac{1}{d+1}$, see Lemma \ref{lem:geo:diam}, it follows directly that the diameter of the visibility region can be bounded by
\begin{align*}
	\rho := \diameter\rbras*{\Vis_z(t)} \ll \rbras*{\frac{\log(t)}{t}}^{\frac{1}{d+1}}.
\end{align*}
Let $D(z, \rho)$ be the set of all points on the boundary $\partial K$ which are of distance at most $\rho$ to $z$, i.e. 
\begin{align*}
	D(z,\rho) = \cbras*{y \in \partial K : \norm*{y - z} < \rho}
\end{align*}
and denote the cap, that is given by the convex hull of $D(z,\rho)$ by $C$, i.e.
\begin{align}\label{eq:cap-contains-vis}
	C := \conv\cbras*{D(z,\rho)}.
\end{align}

By construction, we have $K^x_t \setminus K_t \subseteq \Vis_z(t) \subseteq C$.
It follows from Lemma \ref{lem:geo:vol}, that the volume of $C$ is of order at most $\frac{\log(t)}{t}$.

Fix a linear subspace $\bbL \in G(d,j)$, then one has that the set-difference of the projection of $K^x_t$ and $K_t$ onto the subspace $\bbL$ is contained in the projection of $C$ onto $\bbL$,
\begin{align*}
	(K^x_t|\bbL) \setminus (K_t|\bbL) \subseteq C | \bbL.
\end{align*}
The $j$-dimensional volume of the projected cap $C | \bbL$ can be bounded in its order of magnitude by
\begin{align}\label{eq:vol-cap}
	\Lambda_j(C| \bbL) \ll \rbras*{\tfrac{\log(t)}{t}}^{\tfrac{j+1}{d+1}}.
\end{align}

Depending on the angle between $z$ and $\bbL$, $\sphericalangle(z,\bbL)$, the part $K^x_t \setminus K_t$ is not visible for the orthogonal projection on $\bbL$ since it is hidden behind $K(v \geq \varepsilon_t)$, i.e.
\begin{align*}
	(K^x_t \setminus K^x) | \bbL \subseteq K(v \geq \varepsilon_t) | \bbL,
\end{align*}
for sufficiently large $t$.
To obtain a bound on the maximal angle $\sphericalangle(z,\bbL)$ where the projection does not vanish we approximate $K$ by a ball $\BB^d(z_c,r)$ with center $z_c$ and radius $r$ such that $\BB^d(z_c,r) \subseteq K$ and $\BB^d(z_c,r) \cap \partial K = \cbras*{z}$. Indeed we approximate the boundary $\partial K$ of $K$ from the inside of $K$ by a ball, which is possible, since $K$ is sufficiently smooth.

We repeat the construction of the cap $C$ for $\BB^d(z_c,r)$ with the corresponding $\varepsilon_t$-floating body $\BB^d(v \geq \varepsilon_t)$ of the ball to obtain the cap $C_\BB$ and define $\alpha$ to be the central angle of $C_\BB$ in $\BB^d(z_c,r)$.
It follows from Lemma \ref{lem:geo:vol}, that the volume of $C_\BB$ is of order at most $\frac{\log(t)}{t}$, since $\rho_\BB \ll \rbras{\frac{\log(t)}{t}}^{\frac{1}{d+1}}$, which yields
\begin{align}\label{eq:angle-alpha-bound}
	\alpha \ll \rbras*{\frac{\log(t)}{t}}^{\frac{1}{d+1}}.
\end{align}

Thus it follows from $\BB^d(v \geq \varepsilon_t) \subseteq K(v \geq \varepsilon_t) \subseteq K_t \subseteq K_t^x$ that $K_t^x | \bbL = K_t | \bbL$ if $\sphericalangle(z,\bbL)$ is of larger order than $\alpha$, therefore we have
\begin{align*}
	\Lambda_j\rbras*{(K^x_t | \bbL) \setminus (K_t | \bbL )} \neq 0, 
	\quad \text{only if} \quad 
	\sphericalangle(z,\bbL) \ll \alpha.
\end{align*}

Using \eqref{eq:angle-alpha-bound} and \eqref{eq:vol-cap} it follows that
\begin{align*}
	\Lambda_j\rbras*{(K^x_t | \bbL) \setminus (K_t | \bbL )} & \leq \1\cbras*{\sphericalangle(z,\bbL) \ll \rbras*{\frac{\log(t)}{t}}^{\frac{1}{d+1}}} \Lambda_j(C | \bbL)\\
	& \leq \1\cbras*{\sphericalangle(z,\bbL) \ll \rbras*{\frac{\log(t)}{t}}^{\frac{1}{d+1}}} \rbras*{\frac{\log(t)}{t}}^{\frac{j+1}{d+1}}.
\end{align*}
Finally we use Kubota's formula \eqref{eq:kubota} together with \eqref{eq:first-order-diff-conditioned} and Lemma \ref{lem:measure-of-angle-set} to obtain
\begin{align*}
	& D_xV_j(K_t) = \1\cbras*{x \in K \setminus K(v \geq \varepsilon_t)} c(d,j) \int\limits_{G(d,j)} \Lambda_j\rbras*{(K^x_t | \bbL) \setminus (K_t | \bbL )} \nu_j(\id \bbL)\\
		& \ll \1\cbras*{x \in K \setminus K(v \geq \varepsilon_t)} \int\limits_{G(d,j)} \1\cbras*{\sphericalangle(z,\bbL) \ll \rbras*{\frac{\log(t)}{t}}^{\frac{1}{d+1}}} \rbras*{\frac{\log(t)}{t}}^{\frac{j+1}{d+1}}
		\nu_j(\id \bbL)\\
		& \ll \1\cbras*{x \in K \setminus K(v \geq \varepsilon_t)} \rbras*{\frac{\log(t)}{t}}^{\frac{j+1}{d+1}} \nu_j\rbras*{\cbras*{\bbL \in G(d,j) : \sphericalangle(z,\bbL) \ll \rbras*{\frac{\log(t)}{t}}^{\frac{1}{d+1}}} }\\
		& \ll \1\cbras*{x \in K \setminus K(v \geq \varepsilon_t)}\rbras*{\frac{\log(t)}{t}}^{\frac{j+1}{d+1}}\rbras*{\frac{\log(t)}{t}}^{\frac{d-j}{d+1}} = \1\cbras*{x \in K \setminus K(v \geq \varepsilon_t)}\frac{\log(t)}{t}.
\end{align*}
where $c(d,j) = \binom{d}{j} \frac{\kappa_d}{\kappa_j \kappa_{d-j}}$ can be omitted since we are bounding $D_xV_j(K_t)$ in its order of magnitude with respect to $t$.

\textbf{Second order difference operator:}

Fix $x,y \in K$ and $j \in \cbras*{1,\ldots,d}$, similar to \eqref{eq:first-order-diff-conditioned} and conditioned on the event $A(\varepsilon_t, t)$, we have that
\begin{align*}
	D_{x,y}^2V_j(K_t) = \1\cbras*{x,y \in K \setminus K(v \geq \varepsilon_t)} D_{x,y}^2V_j(K_t).
\end{align*}

To further restrict $x,y \in K \setminus K(v \geq \varepsilon_t)$ we show the following Lemma:
\begin{lemma}
	\label{lem:second-order-diff-volume}
	Fix two convex bodies $P,K \in \cK^d$ with $P \subset K$ and two points $x,y \in K \setminus P$. 
	Denote by $P^{xy}$, $P^{x}$ and $P^{y}$ the convex hulls of $P \cup \cbras*{x,y}$, $P \cup \cbras*{x}$ resp. $P \cup \cbras*{y}$.
	We define the visibility region of $x$ with respect to $P$ by
	\begin{align*}
		\cV_x(P) := \cbras*{z \in K \setminus P : [z,x] \cap P = \emptyset }.
	\end{align*}
	If $\cV_x(P) \cap \cV_y(P) = \emptyset$, then
	\begin{align}
		\label{eq:second-order-diff-volume:set-equation-1}
		P^x \cap P^y & = P,\\
		\label{eq:second-order-diff-volume:set-equation-2}
		P^x \cup P^y & = P^{xy},
	\end{align}
	and further it follows for all valuations $\psi:\cK^d \rightarrow \RR$ that the second order difference operator of $\psi(P)$ vanishes, i.e.
	\begin{align}
		\label{eq:second-order-diff-volume:is-zero}
		D_{x,y}^2\psi(P) = 0.
	\end{align}
\end{lemma}

\begin{proof}
	Using $P^x \subseteq \cV_x(P) \cup P$ and $P^y \subseteq \cV_y(P) \cup P$ it follows directly from $\cV_x(P) \cap \cV_y(P) = \emptyset$ that $P^x \cap P^y \subseteq (\cV_x(P) \cap \cV_y(P)) \cup P = P$. Additionally the inclusion $P \subseteq P^x \cap P^y$, that follows directly from the definition of the convex hull, gives \eqref{eq:second-order-diff-volume:set-equation-1}.
	
	Again, it is immediate that $P^x \subseteq P^{xy}$ and $P^y \subseteq P^{xy}$, thus it remains to prove that $P^{xy} \subseteq P^x \cup P^y$.
	Assume $z \in P^{xy} \setminus (P^x \cup P^y)$, then there existing $\lambda_1, \lambda_2 \in [0,1]$ and $u \in P^x$, $v \in P^y$ such that
	\begin{align*}
		z = \lambda_1 u + (1-\lambda_1)y = \lambda_2 v + (1-\lambda_2) x,
	\end{align*}
	where we can safely assume that $u$ and $v$ are chosen such that $\lambda_1$ and $\lambda_2$ are maximized.
	Note that $u \in P^y$ implies $z \in P^y$ resp. $v \in P^x$ implies $z \in P^x$, a contradiction, which leads to the remaining case $u \in P^x \setminus P^y$ and $v \in P^y \setminus P^x$.
	By construction it now follows that $[x,z] \cap P = \emptyset$ and $[y,z] \cap P = \emptyset$, which yields $z \in \cV_x(P)$ and $z \in \cV_y(P)$ contrary to $\cV_x(P) \cap \cV_y(P) = \emptyset$ which gives \eqref{eq:second-order-diff-volume:set-equation-2}.
	
	The second order difference operator of $\psi(P)$ is given by
	\begin{align*}
		D^2_{x,y}\psi(P) = \psi(P^{xy}) - \psi(P^x) - \psi(P^y) + \psi(P).
	\end{align*}
	Using \eqref{eq:second-order-diff-volume:set-equation-1} and \eqref{eq:second-order-diff-volume:set-equation-2} we can rewrite the first term according to the valuation property \eqref{eq:valuation-property} to
	\begin{align*}
		\psi(P^{xy}) = \psi(P^x) + \psi(P^y) - \psi(P^x \cap P^y) = \psi(P^x) + \psi(P^y) - \psi(P),
	\end{align*}
	which gives \eqref{eq:second-order-diff-volume:is-zero} when substituted in the representation of $D^2_{x,y}\psi(P)$.
\end{proof}

Since $\Vis_x(t) \cap \Vis_y(t) = \emptyset$ implies the conditions of Lemma \ref{lem:second-order-diff-volume} for $P = K_t \supseteq K(v \geq \varepsilon_t)$ it follows that
\begin{align*}
	D_{x,y}^2V_j(K_t) = \1\cbras*{x,y \in K \setminus K(v \geq \varepsilon_t)}\1\cbras*{\Vis_x(t) \cap \Vis_y(t) \neq \emptyset} D_{x,y}^2V_j(K_t).
\end{align*}

Taking $D_{x,y}^2V_j(K_t) = D_x(D_yV_j(K_t))$ we obtain $\abs*{D_{x,y}^2V_j(K_t)} \leq D_xV_j(K^y_t) + D_xV_j(K_t)$
where we immediately see, that the second term $D_xV_j(K_t)$ is the first order difference operator that we have bounded before.
Using $K(v \geq \varepsilon_t) \subseteq K_t \subseteq K_t^y$ we can substitute $K_t$ with $K_t^y$ in the proof for the first order difference operator to obtain
\begin{align*}
	& D_{x,y}^2V_j(K_t) = \1\cbras*{x,y \in K \setminus K(v \geq \varepsilon_t)} \1\cbras*{\Vis_x(t) \cap \Vis_y(t) \neq \emptyset} \rbras*{D_xV_j(K^y_t) + D_xV_j(K_t)}\\
	&  \ll \1\cbras*{x,y \in K \setminus K(v \geq \varepsilon_t)}  \1\cbras*{\Vis_x(t) \cap \Vis_y(t) \neq \emptyset} \rbras*{\frac{\log(t)}{t} + \frac{\log(t)}{t}}\\
	&  \ll \1\cbras*{x,y \in K \setminus K(v \geq \varepsilon_t)}  \1\cbras*{\Vis_x(t) \cap \Vis_y(t) \neq \emptyset} \frac{\log(t)}{t}.
\end{align*}

The results of the prior discussion can be summarized in the following lemma bounding the order of magnitude of the $p$-th absolute moment of the first and second order difference operator of the intrinsic volumes $V_j(K_t)$.

\begin{lemma}
	\label{lem:diff-moments-volume}
	Let $p \in \cbras*{1,\ldots,8}$, $j \in \cbras*{0,\ldots,d}$ and $x,y \in K$, then
	\begin{align}
		\label{eq:diff-moments-volume:first}
		\Eabs*{(D_xV_j(K_t))^p} & \ll \1\cbras*{x \in K \setminus K(v \geq \varepsilon_t)} \rbras*{\frac{\log(t)}{t}}^p,\\
		\label{eq:diff-moments-volume:second}
		\Eabs*{(D_{x,y}^2V_j(K_t))^p} & \ll \1\cbras*{x,y \in K \setminus K(v \geq \varepsilon_t)}\1\cbras*{\Vis_x(t) \cap \Vis_y(t) \neq \emptyset} \rbras*{\frac{\log(t)}{t}}^p,
	\end{align}
\end{lemma}

\begin{proof}
	Let $j \neq 0$. On the event $A(\varepsilon_t, t)$ we use the bounds derived before and on the complementary event $A^c(\varepsilon_t, t)$ it is sufficient to use the estimations $D_xV_j(K_t) \leq V_j(K)$ resp. $\abs*{D_{x,y}^2V_j(K_t)} \leq 2 V_j(K)$, thus
	\begin{align*}
		\Eabs*{(D_xV_j(K_t))^p} & = \E*{(D_xV_j(K_t))^p\1_{A(\varepsilon_t, t)}} + \E*{(D_xV_j(K_t))^p\1_{A^c(\varepsilon_t, t)}}\\
			& \ll  \1\cbras*{x \in K \setminus K(v \geq \varepsilon_t)} \rbras*{\frac{\log(t)}{t}}^p \Prob*{A(\varepsilon_t, t)}\\
			& \quad + V_j(K)^p \Prob*{A^c(\varepsilon_t, t)}
	\end{align*}
	and
	\begin{align*}
		& \Eabs*{(D_{x,y}^2V_j(K_t))^p} \\
		\leq &~ \Eabs*{(D_{x,y}^2V_j(K_t))^p\1_{A(\varepsilon_t, t)}} + \Eabs*{(D_{x,y}^2V_j(K_t))^p\1_{A^c(\varepsilon_t, t)}}\\
		\ll &~ \1\cbras*{x,y \in K \setminus K(v \geq \varepsilon_t)}  \1\cbras*{\Vis_x(t) \cap \Vis_y(t) \neq \emptyset} \rbras*{\frac{\log(t)}{t}}^p \Prob*{A(\varepsilon_t, t)}\\
			& \quad + (2V_j(K))^p \Prob*{A^c(\varepsilon_t, t)}.
	\end{align*}
	Since $\Prob*{A(\varepsilon_t, t)} \leq 1$ and $\Prob*{A^c(\varepsilon_t, t)} \leq t^{-\beta}$, with $\beta = 16d+1 > p$, see Lemma \ref{lem:floating-body-contained}, our claim follows for $j \in \cbras*{1,\ldots,d}$.
	
	Let $j = 0$. We use $V_0(K_t) = \1\cbras*{K_t \neq \emptyset}$ to derive 
	\begin{align*}
		\Eabs*{(D_xV_j(K_t))^p} & = \1\cbras*{x \in K \setminus K(v \geq \varepsilon_t)} \Prob*{K_t = \emptyset}
	\end{align*}
	and
	\begin{align*}
		\Eabs*{(D_{x,y}^2V_0(K_t))^p} & = \1\cbras*{x \in K \setminus K(v \geq \varepsilon_t)}\1\cbras*{\Vis_x(t) \cap \Vis_y(t) \neq \emptyset} \Prob*{K_t = \emptyset},
	\end{align*}
	thus using $\Prob*{K_t = \emptyset} = e^{-t\Lambda_d(K)}$ the claim follows by bounding the exponential decay with $(\frac{\log(t)}{t})^p$.
\end{proof}

Our next objective is to prove corresponding bounds on the moments of the first and second order difference operator of the valuation functional we wish to study. 

\begin{lemma}
	\label{lem:diff-moments-valuations}
	Let $p \in \cbras*{1,\ldots,8}$ and $x,y \in K$, then
	\begin{align}
		\label{eq:diff-moments-valuations:first}
		\Eabs*{(D_x\varphi(K_t))^p} & \ll \1\cbras*{x \in K \setminus K(v \geq \varepsilon_t)} \rbras*{\frac{\log(t)}{t}}^p,\\
		\label{eq:diff-moments-valuations:second}
		\Eabs*{(D_{x,y}^2\varphi(K_t))^p} & \ll \1\cbras*{x,y \in K \setminus K(v \geq \varepsilon_t)}\1\cbras*{\Vis_x(t) \cap \Vis_y(t) \neq \emptyset} \rbras*{\frac{\log(t)}{t}}^p.
	\end{align}
\end{lemma}

\begin{proof}
	Since $D_x$ is linear we obtain
	\begin{align*}
		\Eabs*{(D_x\varphi(K_t))^p} \leq \sum\limits_{j_1, \ldots, j_p = 0}^d \rbras*{\prod\limits_{k=1}^p c_{j_k}} \Eabs*{\prod\limits_{k=1}^p D_xV_{j_k}(K_t)}.
	\end{align*}
	and the generalized H\"older inequality, yields
	\begin{align*}
		\Eabs*{(D_x\varphi(K_t))^p} \leq \sum\limits_{j_1, \ldots, j_p = 0}^d \rbras*{\prod\limits_{k=1}^p c_{j_k}} \rbras*{\prod\limits_{k=1}^p \Eabs*{D_xV_{j_k}(K_t)}^p}^{\frac{1}{p}}.
	\end{align*}
	
	Therefore we can use \eqref{eq:diff-moments-volume:first} to bound all moments on the right hand side yielding \eqref{eq:diff-moments-valuations:first}.
	The proof of \eqref{eq:diff-moments-valuations:second} is similar using \eqref{eq:diff-moments-volume:second} instead of \eqref{eq:diff-moments-volume:first}.
\end{proof}

Before we apply the previous lemma to the error bounds $\tau_1$, $\tau_2$ and $\tau_3$ we introduce two estimations for the domain of integration.
From \cite[Theorem 6.3]{Barany2008} it follows directly that
\begin{align}\label{eq:missed-volume-bound}
	\Lambda_d(K \setminus K(v \geq \varepsilon_t)) \ll \rbras*{\frac{\log(t)}{t}}^{\frac{2}{d+1}}.
\end{align}
Denote by $C(x)$ resp. $C(z)$ the caps constructed according to \eqref{eq:cap-contains-vis} for points $x, z \in K \setminus K(v \geq \varepsilon_t)$, then for every fixed $x \in K \setminus K(v \geq \varepsilon_t)$ we have
\begin{align*}
	\cbras*{y \in K \setminus K(v \geq \varepsilon_t) : \Vis_x(t) \cap \Vis_y(t) \neq \emptyset} \subseteq \bigcup\limits_{z \in \Vis_x(t)} \Vis_z(t) \subseteq \bigcup\limits_{z \in C(x)} C(z).
\end{align*}
Recall that the volumes of $C(x)$ and $C(z)$ are of order at most $\frac{\log(t)}{t}$, thus Lemma \ref{lem:geo:union} yields
\begin{align}\label{eq:visible-intersection-bound}
	\Lambda_d\rbras*{\cbras*{y \in K \setminus K(v \geq \varepsilon_t) : \Vis_x(t) \cap \Vis_y(t) \neq \emptyset}} \ll \frac{\log(t)}{t},
\end{align}
for all $x \in K \setminus K(v \geq \varepsilon_t)$.

Applying the previous results to the error bound $\tau_1$ yields

\begin{align*}
	\tau_1 & \ll \V*{\varphi(K_t)}^{-2} \rbras*{\frac{\log(t)}{t}}^4 t^3 \int\limits_{K} \1\cbras*{x_3 \in K \setminus K(v \geq \varepsilon_t)} \\
			& \quad \times \rbras*{\int\limits_K \1\cbras*{\Vis_{x_1}(t) \cap \Vis_{x_3}(t) \neq \emptyset} \id x_1} \\
			& \quad \times \rbras*{\int\limits_K \1\cbras*{\Vis_{x_2}(t) \cap \Vis_{x_3}(t) \neq \emptyset} \id x_2} \id x_3\\
			& \ll \rbras*{t^{-1-\frac{2}{d+1}}}^{-2} \rbras*{\frac{\log(t)}{t}}^4 t^3 \rbras*{\frac{\log(t)}{t}}^{\frac{2}{d+1}} \rbras*{\frac{\log(t)}{t}}^2\\
			& \ll t^{-1 + \frac{2}{d+1}} \log(t)^{6+\frac{2}{d+1}}.
\end{align*}
In the same manner we can see that
\begin{align*}
	\tau_2 & \ll t^{-1+\frac{2}{d+1}} \log(t)^{6 + \frac{2}{d+1}},\\
	\tau_3 & \ll t^{-\frac{1}{2}+\frac{1}{d+1}} \log(t)^{3+\frac{2}{d+1}}.
\end{align*}

Combining these three bounds with Theorem \ref{thm:bound-uni} leads to
\begin{align*}
	\dist_W(\tilde{\varphi}(K_t),Z) \leq 2 \sqrt{\tau_1} + \sqrt{\tau_2} + \tau_3 \ll t^{-\frac{1}{2}+\frac{1}{d+1}} \log(t)^{3+\frac{2}{d+1}},
\end{align*}
for $Z \distribute \cN(0,1)$, completing the proof of Theorem \ref{thm:univariate}.

\subsection{Proof of Theorem \ref{thm:multivariate}: multivariate functional}

We start to investigate the moments of the first and second order difference operators applied to the components of the $\bff$-vector by combining combinatorial results from \cite{Reitzner2005b} with the floating body and economic cap covering approach from \cite{Reitzner2010} and \cite{ThaeleTurchiWespi2018}.

\textbf{First order difference operator:}

Fix $x \in K$ and $j \in \cbras*{0,\ldots,d-1}$, then conditioned on the event $A(\varepsilon_t, t)$, it follows similar to \eqref{eq:first-order-diff-conditioned} that
\begin{align*}
	D_xf_j(K_t) = \1\cbras*{x \in K \setminus K(v \geq \varepsilon_t)}D_xf_j(K_t),
\end{align*}
thus we can restrict the following to the case $x \in K \setminus K(v \geq \varepsilon_t)$.

Let $K_t$ be fixed and assume $x \not\in K_t$.
Since the polytope $K_t$ is simplicial and all verticies are in general position almost surely, analysis similar to that in \cite[Section 4]{Reitzner2005b} allows us to decompose $D_xf_j(K_t)$ into the number of $j$-faces gained, denoted by $f_j^+$, and the number of $j$-faces lost, denoted by $f_j^-$:
\begin{align*}
	\abs*{D_xf_j(K_t)} = \abs*{f_j^+ - f_j^-} \leq f_j^+ + f_j^-.
\end{align*}
Every $j$-face gained in $K_t^x$ is the convex hull of $x$ and a $(j-1)$-face in $\link(K_t^x,x)$. Additionally it can easily be seen that every $(j-1)$-face in $\link(K_t^x,x)$ is also contained in $\cF_{j-1}^{\Vis}(K_t,x)$, thus
\begin{align*}
	f_j^+ \leq f_{j-1}\rbras*{\link(K_t^x,x)} \leq \abs*{\cF_{j-1}^{\Vis}(K_t,x)}
\end{align*}
On the other hand, the $j$-faces in $\cF_j(K_t)$ that are lost have to be visible from $x$, thus
\begin{align*}
	f_j^- \leq \abs*{\cF_{j}^{\Vis}(K_t,x)}.
\end{align*}
Note that not all visible $j$-faces are removed, to gain the exact amount of lost $j$-face, one has to calculate the number of $j$-faces that are visible and not contained in the link of $x$ in the new polytope $K_t^x$, i.e. $\abs*{\cF_j^{\Vis}(K_t,x) \setminus \link(K_t^x,x)}$, as we will see later for the second order difference operator.

Let $z$ be the closest point to $x$ on the boundary of $\partial K$, then it follows immediately that every visible $i$-face $\frF \in \cF_i^{\Vis}(K_t,x)$ has to be as subset of $\Vis_z(t)$. Since $(i+1)$ pairwise distinct points are needed to form an $i$-face, the number of visible $i$-faces can be bound by the number of points in $\Vis_z(t)$, i.e.
\begin{align*}
	\abs*{\cF_i^{\Vis}(K_t,x)} \leq \binom{\eta\rbras*{\Vis_z(t)}}{i+1}
\end{align*}
for all $i \in \cbras*{0, \ldots, d-1}$.
Combining these steps we obtain
\begin{align*}
	\abs*{D_xf_j(K_t)} \leq \binom{\eta\rbras*{\Vis_z(t)}}{j} + \binom{\eta\rbras*{\Vis_z(t)}}{j+1} = \binom{\eta\rbras*{\Vis_z(t)}+1}{j+1} \leq \rbras*{\eta\rbras*{\Vis_z(t)}+1}^{j+1}
\end{align*}
and further for $p \in \cbras*{1,\ldots,8}$,
\begin{align*}
	\Eabs*{D_xf_j(K_t)\1_{A(\varepsilon_t,t)}}^p \leq \E*{\rbras*{\eta\rbras*{\Vis_z(t)}+1}^{p(j+1)}}.
\end{align*}
The binomial theorem and the fact that $\eta\rbras*{\Vis_z(t)}$ is Poisson distributed with parameter $\mu(\Vis_z(t))$ yields
\begin{align*}
	\Eabs*{D_xf_j(K_t)\1_{A(\varepsilon_t,t)}}^p & \leq \sum\limits_{i = 0}^{p(j+1)} \binom{p(j+1)}{i} \E*{\eta\rbras*{\Vis_z(t)}^i}\\
		& \leq \sum\limits_{i=0}^{p(j+1)} \binom{p(j+1)}{i} \sum\limits_{k=0}^i \stirling{i}{k} \mu(\Vis_z(t))^k
\end{align*}
where $\stirling{i}{k}$ denotes the Stirling numbers of the second kind.
Recall that $\mu(\Vis_z(t)) = t \Lambda_d(\Vis_z(t)) \ll t \frac{\log(t)}{t} = \log(t)$ and $j \in \cbras*{0, \ldots, d-1}$, thus
\begin{align*}
	\Eabs*{D_xf_j(K_t)\1_{A(\varepsilon_t,t)}}^p  \ll \log(t)^{pd}.
\end{align*}

Conditioned on the complementary event $A^c(\varepsilon_t,t)$ we need to slightly modify the proof, replacing $\eta(\Vis_z(t))$ by $\eta(K)$, the number of all points in $K$, and using the Cauchy-Schwarz inequality to separate the expectation of the indicator, from the moments of $\eta(K)$:

\begin{align*}
	\Eabs*{D_xf_j(K_t)\1_{A^c(\varepsilon_t,t)}}^p & \leq \sum\limits_{i=0}^{p(j+1)} \binom{p(j+1)}{i} \rbras*{ \E*{\1_{A^c(\varepsilon_t,t)}} \E*{\eta(K)^{2i}}}^\frac{1}{2}\\
		& \leq \sqrt{\Prob*{A^c(\varepsilon_t,t)}} \sum\limits_{i=0}^{p(j+1)} \binom{p(j+1)}{i} \rbras*{\sum\limits_{k=0}^{2i} \stirling{i}{k} \mu(K)^k}^\frac{1}{2}\\
		& \ll t^{-\frac{\beta}{2}} t^{pd} \ll 1,
\end{align*}
since $\beta = 16d+1 \geq 2pd$.

\textbf{Second order difference operator}:

Fix $x,y \in K$ and $j \in \cbras*{0, \ldots, d-1}$, similar to the intrinsic volumes handled before we have
\begin{align*}
	D_{x,y}^2f_j(K_t) = \1\cbras*{x,y \in K \setminus K(v \geq \varepsilon_t)}D_{x,y}^2 V_j(K_t).
\end{align*}
We show the following Lemma, to obtain a restriction on $x,y$ for the components of the $\bff$-vector similar to that derived from Lemma \ref{lem:second-order-diff-volume} for the intrinsic volumes:

\begin{lemma}\label{lem:second-oder-diff-f-vector}
	Fix a $d$-dimensional polytope $P \subset K$ contained in a convex body $K \in \cK^d$ and two points $x,y \in K \setminus P$. Let $\cV_x(P)$ and $\cV_y(P)$ denote the visibility regions of $x$ and $y$ with respect to $P$ be defined as in Lemma \ref{lem:second-order-diff-volume}.
	If $\cV_x(P) \cap \cV_y(P) =\emptyset$, then
	\begin{align*}
		D_{x,y}^2f_j(P) = 0,
	\end{align*}
	for all $j \in \cbras*{0, \ldots, d-1}$.
\end{lemma}

\begin{proof}
	Denote by $P^{xy}$, $P^x$ and $P^y$ the convex hulls of $P \cup \cbras*{x,y}$, $P \cup \cbras*{x}$ resp. $P \cup \cbras*{y}$, then we can decompose the number of $j$-faces of $P^{xy}$ into the number of $j$-faces of $P^x$ and the gained and lost $j$-faces obtained from adding $y$, i.e. 
	\begin{align*}
		f_j(P^{xy}) = f_j(P^x) + f_{j-1}(\link(P^{xy},y)) - \abs*{\cF_j^{\Vis}(P^x,y) \setminus \link(P^{xy},y)}
	\end{align*}
	Since the visible regions are disjoint we have $\cF^{\Vis}(P^x,y) = \cF^{\Vis}(P,y)$ and $\link(P^{xy},y) = \link(P^y,y)$, thus
	\begin{align*}
		D_{x,y}^2f_j(P) & = f_j(P^{xy}) - f_j(P^x) - f_j(P^y) + f_j(P)\\
	 		& = f_{j-1}(\link(P^{y},y)) - \abs*{\cF_j^{\Vis}(P,y) \setminus \link(P^{y},y)} - f_j(P^y) + f_j(P).
	\end{align*}
	Similar to $f_j(P^{xy})$ we can decompose $f_j(P)$ by counting the $j$-faces in $P^y$ and subtracting the difference that arises from the addition of $y$ to $P$:
	\begin{align*}
		f_j(P) = f_j(P^y) - f_{j-1}(\link(P^y, y)) + \abs*{\cF_j^{\Vis}(P,y) \setminus \link(P^y,y)},
	\end{align*}
	yielding 
	\begin{align*}
		D_{x,y}^2f_j(P) & = f_{j-1}(\link(P^{y},y)) - \abs*{\cF_j^{\Vis}(P,y) \setminus \link(P^y,y)} - f_j(P^y)\\
		& \quad + f_j(P^y) - f_{j-1}(\link(P^y, y)) + \abs*{\cF_j^{\Vis}(P,y) \setminus \link(P^y,y)}\\
			& = 0,
	\end{align*}
	which is the desired conclusion.
\end{proof}

Since $\Vis_x(t) \cap \Vis_y(t) = \emptyset$ implies the conditions of Lemma \ref{lem:second-order-diff-volume} for $P = K_t \supseteq K(v \geq \varepsilon_t)$ it follows that
\begin{align*}
	D_{x,y}^2f_j(K_t) = \1\cbras*{x,y \in K \setminus K(v \geq \varepsilon_t)}\1\cbras*{\Vis_x(t) \cap \Vis_y(t) \neq \emptyset} D_{x,y}^2f_j(K_t),
\end{align*}
and similar to $D_{x,y}^2V_j(K_t)$ we derive
\begin{align*}
	D_{x,y}^2(K_t) \ll \1\cbras*{x,y \in K \setminus K(v \geq \varepsilon_t)}\1\cbras*{\Vis_x(t) \cap \Vis_y(t) \neq \emptyset} \log(t)^{pd}.
\end{align*}

Having disposed of this preliminary steps, we can now summarize the results in the following Lemma bounding the order of magnitude of the $p$-th moment of the difference operator of the $\bff$-vector components $f_j(K_t)$.

\begin{lemma}\label{lem:diff-moments-f}
	Let $p \in \cbras*{1,\ldots,8}$, $j \in \cbras*{0,\ldots,d-1}$ and $x,y \in K$, then
	\begin{align*}
		\Eabs*{(D_xf_j(K_t))^p} & \ll \1\cbras*{x \in K \setminus K(v \geq \varepsilon_t)} \log(t)^{dp},\\
		\Eabs*{(D_{x,y}^2f_j(K_t))^p} & \ll \1\cbras*{x,y \in K \setminus K(v \geq \varepsilon_t)}\1\cbras*{\Vis_x(t) \cap \Vis_y(t) \neq \emptyset} \log(t)^{dp}.
	\end{align*}
\end{lemma}

\begin{proof}
	The proof is similar to the one of Lemma \ref{lem:diff-moments-volume}.
\end{proof}

We are left with the task of applying our estimations on the bound $\gamma_1$, $\gamma_2$ and $\gamma_3$ given by Theorem \ref{thm:bound-multi}.
Since we consider the multivariate functional given by \eqref{eq:multivariate-functional} we have to distinguish the following three cases depending on the combination of functionals $F_i$ and $F_j$ using the corresponding bounds for the variance given by \eqref{eq:intrinsic:variance} for the intrinsic volumes and by
\begin{align}\label{eq:fvec:variance}
	t^{1-\frac{2}{d+1}} \ll \V*{f_k(K_t)} \ll t^{1-\frac{2}{d+1}},
\end{align}
for the components of the $\bff$-vector, see \cite{Reitzner2005a}.
We denote by $\gamma_1(i,j)$, $\gamma_2(i,j)$ resp. $\gamma_3(i)$ the integral in $\gamma_1$, $\gamma_2$ resp. $\gamma_3$, then it follows from Lemma \ref{lem:diff-moments-volume}, \ref{lem:diff-moments-f} and the estimations on the domain of integration \eqref{eq:missed-volume-bound} and \eqref{eq:visible-intersection-bound}, that
\begin{align*}
	\gamma_1(i,j), \gamma_2(i,j) \ll t^{-1+\frac{2}{d+1}} \times 
		\begin{cases}
			 \log(t)^{6+\frac{2}{d+1}}, & i,j \in \cbras*{1,\ldots,d},\\
			\log(t)^{4+2d+\frac{2}{d+1}}, & i \in \cbras*{1,\ldots,d}, j \in \cbras*{d+1, \ldots, 2d},\\
			\log(t)^{2+4d+\frac{2}{d+1}}, & i,j \in \cbras*{d+1, \ldots, 2d}.
		\end{cases}
\end{align*}
and
\begin{align*}
	\gamma_3(i) \ll t^{-\frac{1}{2}+\frac{1}{d+1}} \times 
		\begin{cases}
				\log(t)^{3+\frac{2}{d+1}}, & i \in \cbras*{1, \ldots, d},\\
				\log(t)^{3d + \frac{2}{d+1}}, & i \in \cbras*{d+1, \ldots, 2d}.
		\end{cases}
\end{align*}
It can easily be seen that the speed of convergence is dominated by the case $i,j \in \cbras*{d+1, \ldots, 2d}$ thus we can rewrite the bound in Theorem \ref{thm:bound-multi} using $\sigma_{ij}(t) = \Cov*{F_i}{F_j}$ to
\begin{align*}
	& \dist_3(F,N_\Sigma(t)) \ll d \sum\limits_{i,j=1}^{2d} \abs*{\sigma_{ij}(t) - \Cov*{F_i}{F_j}}\\
	& + 2d \cdot t^{-\frac{1}{2}+\frac{1}{d+1}}\log(t)^{1+2d+\frac{1}{d+1}} + d \cdot t^{-\frac{1}{2}+\frac{1}{d+1}}\log(t)^{1+2d+\frac{1}{d+1}} + d^2 \cdot t^{-\frac{1}{2}+\frac{1}{d+1}}\log(t)^{3d+\frac{2}{d+1}}\\
	& \ll t^{-\frac{1}{2}+\frac{1}{d+1}}\log(t)^{3d+\frac{2}{d+1}},
\end{align*}
which completes the proof.

\subsection{Proof of Theorem \ref{thm:oracle}: oracle functional}

Recall that the oracle estimator is given by $\hat{\vartheta}_{oracle}(K_t) = V_d(K_t) + t^{-1} f_0(K_t)$, and its variance asymptotics is given by \eqref{eq:oracle-variance},
\begin{align*}
	\V*{\hat{\vartheta}_{oracle}(K_t)} = \gamma_d \Omega(K)(1+o(1))t^{-1-\frac{2}{d+1}} = \Theta\rbras*{t^{-1-\frac{2}{d+1}}},
\end{align*}
for $t \rightarrow \infty$, where the constant $\gamma_d$ only depends on the dimension and is known explicitly and $\Omega(K)$ denotes the affine surface area of $K$. 
Rescaling of $\vartheta(K_t)$ yields
\begin{align*}
	\frac{\hat{\vartheta}(K_t)}{\sqrt{\V{\hat{\vartheta}(K_t)}}} = \Theta\rbras*{t^{\frac{1}{2}+\frac{1}{d+1}} V_d(K_t) + t^{-\frac{1}{2}+\frac{1}{d+1}}f_0(K_t)},
\end{align*}
where the scaling $t^{\frac{1}{2}+\frac{1}{d+1}}$ resp. $t^{-\frac{1}{2}+\frac{1}{d+1}}$ corresponds with the asymptotic variance of $V_d(K_t)$ resp. $f_0(K_t)$, see \eqref{eq:intrinsic:variance} and \eqref{eq:fvec:variance}. 
Therefore we can use the previous results to deduce bounds on the moments of the first and second order difference operator of the standardized oracle estimator $\tilde{\vartheta}(K_t)$.

\begin{lemma}
	Let $p \in \cbras*{1,\ldots, 4}$ and $x,y \in K$, then
	\begin{align*}
		\Eabs{(D_x\tilde{\vartheta}(K_t))^p} & \ll \1\cbras*{x \in K \setminus K(v \geq \varepsilon_t)} t^{-\frac{p}{2}+\frac{p}{d+1}} \log(t)^{dp},\\
		\Eabs{(D_{x,y}^2\tilde{\vartheta}(K_t))^p} & \ll \1\cbras*{x,y \in K \setminus K(v \geq \varepsilon_t)} \1\cbras*{\Vis_x(t) \cap \Vis_y(t) \neq \emptyset} t^{-\frac{p}{2} + \frac{p}{d+1}} \log(t)^{dp}.
	\end{align*}
\end{lemma}

\begin{proof}
	Using the binominal theorem and the Cauchy–Schwarz inequality it follows directly with Lemma \ref{lem:diff-moments-volume} and Lemma \ref{lem:diff-moments-f} that
	\begin{align*}
		& \Eabs*{t^{\frac{1}{2}+\frac{1}{d+1}} D_xV_d(K_t) + t^{-\frac{1}{2}+\frac{1}{d+1}}D_xf_0(K_t)}^p\\
		& \leq \sum\limits_{j=0}^p \binom{p}{j} \rbras*{ \Eabs*{t^{\frac{1}{2}+\frac{1}{d+1}}D_xV_d(K_t)}^{2j} \Eabs*{t^{-\frac{1}{2}+\frac{1}{d+1}}D_xf_0(K_t)}^{2(p-j)} }^\frac{1}{2}\\
		& \ll \1\cbras*{x \in K \setminus K(v \geq \varepsilon_t)}\sum\limits_{j=0}^p \binom{p}{j}  \rbras*{\rbras*{t^{\frac{1}{2}+\frac{1}{d+1}} \frac{\log(t)}{t}}^{2j} \rbras*{t^{-\frac{1}{2}+\frac{1}{d+1}} \log(t)^d}^{2(p-j)}}^\frac{1}{2}\\
		& = \1\cbras*{x \in K \setminus K(v \geq \varepsilon_t)}t^{-\frac{p}{2}+\frac{p}{d+1}} \sum\limits_{j=0}^p \binom{p}{j} \log(t)^{j+d(p-j)}.
	\end{align*}
	Since $j+d(p-j) \leq dp$ the desired conclusion follows. The proof for the second order difference operator is similar.
\end{proof}

Applying these estimations to the bound $\tau_1$, $\tau_2$ and $\tau_3$ in Theorem \ref{thm:bound-uni} yields
\begin{align*}
	\tau_1 & \ll t^{-1+\frac{2}{d+1}}\log(t)^{2+4d+\frac{2}{d+1}},\\
	\tau_2 & \ll t^{-1+\frac{2}{d+1}}\log(t)^{2+4d+\frac{2}{d+1}},\\
	\tau_3 & \ll t^{-\frac{1}{2}+\frac{1}{d+1}}\log(t)^{3d+\frac{2}{d+1}},
\end{align*}
thus
\begin{align*}
	\dist_W(\tilde{\vartheta}(K_t),Z) \leq 2 \sqrt{\tau_1} + \sqrt{\tau_2} + \tau_3 \ll t^{-\frac{1}{2}+\frac{1}{d+1}}\log(t)^{3d+\frac{2}{d+1}},
\end{align*}
for $Z \distribute \cN(0,1)$, completing the proof of Theorem \ref{thm:oracle}.

\bigskip

\textbf{Acknowledgments.} The author wishes to express his thanks to Matthias Schulte for helpful suggestions on the multivariate limit theorem and to Martina Juhnke-Kubitzke for several helpful comments concerning Lemma \ref{lem:second-order-diff-volume} and \ref{lem:second-oder-diff-f-vector}.

\printbibliography

\listoffixmes

\end{document}